[1994/12/01]% LaTeX date must December 1994 or later
\documentclass[a4paper, 12pt]{amsart}%{article}[amssymb]%
\usepackage{amssymb,amsxtra}
\usepackage{amsmath,amsthm,amscd}

\chardef\bslash=`\\ % p. 424, TeXbook
\newtheorem{thm}{Theorem}[section]
\newtheorem{cor}[thm]{Corollary}
\newtheorem{lem}[thm]{Lemma}
\newtheorem{prop}[thm]{Proposition}

\theoremstyle{definition}
\newtheorem{defn}{Definition}[section]

\theoremstyle{remark}
\newtheorem{rem}{Remark}[section]

\newcommand{\thmref}[1]{Theorem~\ref{#1}}

\newcommand{\lemref}[1]{Lemma~\ref{#1}}
\newcommand{\propref}[1]{Proposition~\ref{#1}}
\newcommand{\corref}[1]{Corollary~\ref{#1}}
\newcommand{\defnref}[1]{Definition~\ref{#1}}

\newcommand{\eval}[2][\right]{\relax
  \ifx#1\right\relax \left.\fi#2#1\rvert}

\def\lsd{\mbox{$\mathop{\mathrel{\vrule height 5pt
 depth 0pt}\joinrel\mathrel\times} $}}
 
 \def\R{{\bf{R}}}
\def\C{{\bf{C}}}

\def\N{{\bf{N}}}

\def\defequal{\stackrel{\mathrm{def}}{=}}

\def\ra{\rightarrow}

\def\tfae{the following conditions are equivalent:}

\title{Continuous bounded cocycles}
\author{Jean Renault}
\address{D\'epartment de Math\'ematiques, Universit\'e d'Orl\'eans,
45067 Orl\'eans, France}
\email{Jean.Renault@univ-orleans.fr}

\usepackage{pdfsync}
\begin{document}
\vskip5mm
\begin{abstract}
Let $G$ be a minimal locally compact groupoid with compact metrizable unit space and let $E$ be a continuous $G$-Hilbert bundle. We show that a bounded continuous cocycle $c: G\ra r^*E$ is necessarily a continuous coboundary. This is a groupoid version of a classical theorem of Gottschalk and Hedlund.\end{abstract} 

\maketitle
\markboth{Jean Renault}
{Bounded groupoid cocycles}

\renewcommand{\sectionmark}[1]{}

\section{Introduction.}

D. Coronel, A. Navas and M. Ponce have recently given in \cite{cnp:bdd orbits} a generalization of the classical Gottschalk-Hedlund theorem on bounded cocycles to affine isometric actions on a Hilbert space. The present work gives a groupoid version of their result. First, instead of a semigroup $\Gamma$ acting by continuous maps on a space $X$, our dynamical system is given by a locally compact groupoid $G$. Second, instead of a single Hilbert space, we consider a $G$-Hilbert bundle. Going from a dynamical system $(\Gamma, X)$ to a topological groupoid $G$ is only a minor variation. The main difficulty will be to pass from a constant Hilbert bundle to a continuous field of Hilbert spaces.

Let us first recall what the Gottschalk-Hedlund theorem \cite[Chapter 14]{gh:top dyn} says.

\begin{thm} {(\bf Gottschalk-Hedlund)}\label{Gottschalk-Hedlund} Let $T$ be a minimal continuous map on a compact space $X$ and let $f:X\ra\C$ be a continuous function. Then the following properties are equivalent:
\begin{enumerate}
\item there exists a continuous function $g:X\ra\C$ such that for all $x\in X$, $f(x)=g(x)-g(Tx)$;
\item there exists $x\in X$ and $M\in\R$ such that for all $n\in\N$,
$$|\sum_{i=0}^n f(T^i x)|\le M;$$
\item there exists $M\in\R$ such that for all $x\in X$ and for all $n\in\N$,
$$|\sum_{i=0}^n f(T^i x)|\le M.$$
\end{enumerate}
\end{thm}

Let us fix a point of terminology: compact means here quasi-compact and Hausdorff; locally compact means that each point has a compact neighborhood. A locally compact space is not necessarily Hausdorff. 
An easy generalisation of this theorem is given in \cite{ren:approach} in the language of groupoids and cocycles. For example, given $(T,X)$ as above, one can consider the topological groupoid
$$G(X,T)=\{(x,m-n,y): x,y\in X,\quad m,n\in\N\quad T^mx=T^ny\}$$
with range and source maps $r,s: G(X,T)\ra X$ given respectively by $r(x,k,y)=x$ and $s(x,k,y)=y$, multiplication $(x,k,y)(y,l,z)=(x,k+l,z)$, inverse map $(x,k,y)^{-1}=(y,-k,x)$ and basic open sets of the form
$${\mathcal U}(U,V,m,n)=\{(x,m-n,y): (x,y)\in U\times V\quad T^mx=T^ny\}.$$
A function $f:X\ra A$, where $A$ is an abelian group defines $c_f: G(X,T)\ra A$ according to
$$c_f(x, m-n, y)=f(x)+f(Tx)+\ldots+f(T^{m-1}x)-f(T^{n-1}y)-\ldots -f(Ty)-f(y).$$
This is a cocycle (with respect to the trivial action of $G$ on $A$): $c_f(\gamma\gamma')=c_f(\gamma)+c_f(\gamma')$ for all composable pairs $(\gamma,\gamma')$.
This cocycle is a coboundary if and only $f$ is a coboundary, i.e. of the form $f=g-g\circ T$: then $c_f=g\circ r-g\circ s$. Let us assume that $A$ is a topological group. The cocycle is continuous if and only if $f$ is continuous; it is said to be a continuous coboundary if the function $g$ can be chosen continuous. Let us recall that a topological groupoid $G$ with unit space $X$ is said to be minimal if $\emptyset$ and $X$ are the only open invariant subsets of $X$.  Here is a groupoid version of the Gottschalk-Hedlund theorem adapted from Theorem 1.4.10 of \cite{ren:approach}.

\begin{thm} \label{LNM 793} Let $G$ be a minimal topological groupoid with compact  unit space $X$ and let $A$ be a topological abelian group without non-trivial compact subgroup. For a continuous cocycle $c:G\ra A$, the following properties are equivalent:
\begin{enumerate}
\item there exists a continuous function $g:X\ra A$ such that $c=g\circ r-g\circ s$,  
\item there exists $x\in X$ such that $c(G_x)$ (where $G_x=s^{-1}(x)$) is relatively compact,
\item  $c(G)$ is relatively compact.
\end{enumerate}
\end{thm}

\begin{rem}
The extra assumption, namely that $G$ admits a cover of continuous $G$-sets, made in \cite{ren:approach} is in fact not needed. However, our standing assumption for topological groupoids is that the range and source maps are open. The above groupoid $G(X,T)$ satisfies this assumption only if $T$ is an open map.
\end{rem}

The above setting is unsatisfactory: the natural data for continuous groupoid cohomology (see for example \cite{tu:cohomology}) consist of:
\begin{itemize}
\item a topological groupoid $G$ over a topological space $X$,
\item  a space of coefficients (or $G$-module) $A$, which is a continuous bundle of topological abelian groups $A_x$ over $X$ endowed with a continuous $G$-action, i.e. $G$ acts by isomorphisms $L(\gamma): A_{s(\gamma))}\ra A_{r(\gamma)}$ and the action map $G*A\ra A$ is continuous.
\end{itemize}
The above theorem only covers the case when $A$ is a constant bundle endowed with a trivial action. We refer the reader to \cite{mrw:Morita} for the definition of a groupoid action. 

In this framework, a (one-)cocycle is a section $c: G\ra r^*A$ (i.e. $c(\gamma)\in A_{r(\gamma)}$) such that $c(\gamma\gamma')=c(\gamma)+L(\gamma)c(\gamma')$. It is a coboundary if there exists a section $f: X\ra A$ such that $c(\gamma)=f\circ r(\gamma)-L(\gamma)f\circ s(\gamma)$. A cocycle $c: G\ra r^*A$ defines an affine action of $G$ on $A$, given by
$$\gamma a=L(\gamma)a+c(\gamma).$$
Let us denote by $A(c)$ this $G$-affine space. More generally, a $G$-affine space over $A$ is a space $Z$ endowed with a left action of $G$ and a right principal action of $A$ (written additively) such that
\begin{itemize}
\item $r: Z\ra G^{(0)}$ identifies $Z/A$ with $G^{(0)}$,
\item $\gamma(z+a)=\gamma z+L(\gamma)a$.
\end{itemize}
The choice of a section $f$ for $r: Z\ra G^{(0)}$ gives a cocycle
$$c(\gamma)=\gamma f\circ s(\gamma)-f\circ r(\gamma)$$
which identifies $Z$ with $A(c)$: $f$ becomes the zero section. Up to a coboundary, the cocycle depends only on the isomorphism class of the $G$-affine space $Z$. A $G$-affine space over $A$ is said to be trivial if it is isomorphic to $A$; this is equivalent to the existence of a $G$-equivariant section $f$.
\vskip 5mm
We are interested in the case when the space of coefficients is a $G$-Hilbert bundle. We shall assume here that our Hilbert spaces are complex Hilbert spaces but our proofs and results hold for real Hilbert spaces as well. The scalar product, which is chosen to be linear in the second variable, is denoted by $(u|v)$ and the associated norm is denoted by $\|u\|$. We shall deal with a family $(E_x)_{x\in X}$ of Hilbert spaces; when there is no ambiguity, we shall omit the index $x$ in $(u|v)_x$ and in $\|u\|_x$. We shall use \cite{dd:dixmier-douady} and \cite{fd:representations} as references to Hilbert bundles. Let us recall the definition of a Hilbert bundle given in \cite[II.13.5]{fd:representations} as a particular case of a Banach bundle:

\begin{defn}\label{Banach bundle} Let $X$ be a topological space. A Banach [resp. Hilbert] bundle over $X$ is a pair $<E,\pi>$ where $E$ is a topological space called the bundle space, and $\pi:E\ra X$ is a continuous open surjection called the bundle projection, together with operations and norms making each fiber $E_x=\pi^{-1}(x)$ into a Banach [resp. Hilbert] space, and satisfying the following conditions:
\begin{enumerate}
\item $u\mapsto \|u\|$ is continuous on $E$ to $\R$.
\item The operation $+$ is continuous as a map on $E*E\defequal \{(u,v)\in E\times E: \pi(u)=\pi(v)\}$ to $E$.
\item For each scalar $\lambda$, the map $u\mapsto \lambda u$ is continuous on $E$ to $E$.
\item If $x\in X$ and $(u_i)$ is a net of elements of $E$ such that $\|u_i\|\to 0$ and $\pi(u_i)\to x$ in $X$, then $u_i\to 0_x$ in $E$, where $0_x$ is the zero element of $E_x$.
\end{enumerate}
\end{defn}
It is assumed in \cite{fd:representations} that $X$ is Hausdorff. We shall make the weaker assumption that $X$ is locally Hausdorff (every point has a Hausdorff neighborhood) and shall apply results of \cite{fd:representations} to the reduction of $E$ to Hausdorff subspaces of $X$. Elements of $E$ will be denoted by $u$ or by $(x,u)$ where $x=\pi(u)$. An important result (see Appendix C of \cite{fd:representations}) says that when $X$ is paracompact, there are sufficiently many continuous sections, i.e. for all $(x,u)\in E$, there is a continuous section $f:X\ra E$ such that $f(x)=u$. The related notion of continuous field of Banach spaces developed in \cite{dd:dixmier-douady} (see also \cite[10.1]{dix:C*}) privileges continuous sections. Both notions -- Banach bundle and continuous field of Banach spaces -- are equivalent when the base space $X$ is paracompact. One can recover the topology of $E$ from the space $C(X,E)$ of continuous sections as follows:  $g\in C(X,E)$, $V$ open subset of $X$ and $\epsilon>0$ define a ``tube''
$$T(g,V,\epsilon)=\{(x,u)\in E: x\in V,\quad \|u-g(x)\|<\epsilon\}.$$
The family of these tubes form a base for the topology of $E$ (\cite[Theorem II.13.18]{fd:representations}). When $X$ is locally paracompact (every point has a paracompact neighborhood), we consider continuous sections $g:V\ra E$ where $V$ is open and paracompact instead of global continuous sections. For the sake of simplicity, we shall always assume that the base space of the bundle is locally compact.

When $E$ is a Hilbert bundle over a locally compact Hausdorff space $X$, the space ${\mathcal E}=C_0(X,E)$ of continuous sections vanishing at infinity is a Hilbert $C_0(X)$-module, where $C_0(X)$ is the C*-algebra of complex-valued continuous functions vanishing at infinity endowed with the sup-norm: given $h\in C_0(X)$ and $f\in C_0(X,E)$, we define $fh\in C_0(X,E)$ by $(fh)(x)=f(x)h(x)$ and given $f,g\in C_0(X,E)$, we define $<f,g>\in C_0(X)$ by $<f,g>(x)=(f(x)|g(x))$. 

Let us turn to the definition of a $G$-Hilbert bundle. From now on $G$ designates a topological groupoid with unit space $X$. 

\begin{defn} A $G$-Hilbert bundle is a Hilbert bundle $\pi:E\ra X$ together with a continuous action $G*E\ra E$, where as usual, $G*E=\{(\gamma, u)\in G\times E: s(\gamma)=\pi(u)\}$, sending $(\gamma, u)$ to $L(\gamma)u$ and such that for all $\gamma \in G$, $L(\gamma): E_{s(\gamma))}\ra E_{r(\gamma)}$ is a linear isometry.\end{defn}

The theory of induced representations provides a justification for studying $G$-Hilbert bundles. For example, if $E$ is a $H$-Hilbert space, where $H$ is a closed subgroup of a locally compact group $G$, then the quotient $(G\times E)/H$, where $H$ acts by the diagonal action $h(g,e)=(gh^{-1}, L(h)e)$ is an equivariant $G$-Hilbert bundle over $G/H$; equivalently, it is a $G\,\lsd G/H$-Hilbert bundle, where $G\,\lsd G/H$ is the groupoid of the left action of $G$ on $G/H$. A trivialization of this bundle would often require a continuous section of the quotient map $G\ra G/H$, which may not exist.

We can now state our theorem. 

\begin{thm}\label{main} Let $G$ be a minimal locally compact groupoid on a compact metrizable space $X$, let $E$ be a continuous $G$-Hilbert bundle and let $c:G\ra r^*H$ be a continuous cocycle. Assume that $E$ is second countable. Then the following properties are equivalent:
\begin{enumerate}
\item $c$ is a continuous coboundary,
\item there exists $x\in X$ such that $\|c(G_x)\|$ is bounded,
\item $\|c(G)\|$ is bounded.
\end{enumerate}
\end{thm}

These properties have a nice interpretation in terms of the $G$-affine space $E(c)$. As we have seen earlier, condition $(i)$ says that $E(c)$ is trivial or, equivalently, admits a $G$-equivariant continuous section. Condition $(ii)$ says that there exists a bounded $G$-orbit in $E(c)$. Condition $(iii)$ says that all $G$-orbits in $E(c)$ are bounded. The implications $(i)\Rightarrow (iii) \Rightarrow (ii)$ are obvious.

When $G$ is a group, this theorem is a well-known result (see for example \cite[Proposition 2.2.9]{bhv:T}). In fact, it is valid for a much larger class of Banach spaces than Hilbert spaces (we still assume that the action is isometric!). U. Bader, T. Gelander and N. Monod have recently shown in \cite{bgm:fixed} that it is true for Banach spaces which are $L$-embedded.
 When $G=G(X,T)$ as above and $E=X\times\C$ with the trivial $G$-action, this is \thmref{Gottschalk-Hedlund}.
The situation studied by Coronel, Navas and Ponce is essentially the case when $E=X\times F$, where $F$ is a fixed Hilbert space, is a constant Hilbert bundle (but on which $G$ acts non-trivially). More precisely, they consider a skew action of a semigroup $\Gamma$ on $X\times F$ where $g\in\Gamma$ acts continuously on $X\times F$ according to
$g(x,v)=(g(x), I(g,x)v)$, where $I(g,x)$ is an isometry of $F$ and satisfies $$I(gh, x)=I(g, h(x))I(h,x).$$
Under additional assumptions on the dynamical system $(\Gamma,X)$, this can be put into the groupoid setting along the lines of \cite{er:semigroups}. 

We shall prove that $(ii)\Rightarrow (i)$, namely the existence of a bounded $G$-orbit in $E(c)$ implies the existence of a $G$-equivariant continuous section. Our proof  is modelled after \cite[Section 4]{cnp:bdd orbits} and consists of two steps. First, we show that $E(c)$ admits a $G$-equivariant weakly continuous section. Secondly, we show that a $G$-equivariant weakly continuous section is automatically continuous.

\section{Existence of a $G$-equivariant weakly continuous section.}

The main task is to define the weak topology on the bundle space $E$ of a continuous Hilbert bundle over a topological space $X$. It is easy to do when $E=X\times F$ is a constant bundle: then we just consider the product topology $X\times F_\sigma$, where  $F_\sigma$ is the Hilbert space $F$ endowed with the weak topology.

\begin{prop} Let $E$ be a Hilbert bundle over a locally compact space $X$ and let $(\underline x,\underline u)\in E$. Then the sets
$$U(V;f_1,\ldots,f_n;\epsilon)=\{(x,u)\in E: x\in V,\, \forall i=1,\ldots, n, |(f_i(x)|u)-(f_i(\underline x)|\underline u)|<\epsilon\},$$
where $V$ is a compact neighborhood of $\underline x$, for all $i=1,\ldots, n\,$,$\,f_i:V\ra E$ is a continuous section and $\epsilon>0$,
form a fundamental system of neighborhoods of $(\underline x,\underline u)$ for a topology of $E$. 
\end{prop}
This topology is called the weak topology of $E$. When $E$ is endowed with the weak topology, it is denoted by $E_\sigma$. The original Hilbert bundle topology is called the strong topology. We let the reader check that the strong topology is finer than the weak topology.
\begin{proof} One checks that the family ${\mathcal V}(\underline x,\underline u)$ of subsets of $E$ containing some $U(V;f_1,\ldots,f_n;\epsilon)$ satisfies the axioms $(V_I), (V_{II}), (V_{III})$ and $(V_{IV})$ of \cite[Section 1.2]{bbki:topologie}.
\end{proof}
One can also check that $E_\sigma=X\times F_\sigma$ when $E=X\times F$, where $F$ is a fixed Hilbert space.

\begin{prop} Let $E$ be a Hilbert bundle over a compact space $X$. Let ${\mathcal E}=C(X,E)$ be the Banach space of continuous sections of $E$, equipped with the sup-norm. Then, the map $(id_X,j): E_\sigma\ra X\times {\mathcal E}^*_\sigma$, where ${\mathcal E}^*_\sigma$ is the dual Banach space ${\mathcal E}^*$ endowed with the $*$-weak topology and where $j: E\ra {\mathcal E}^*$ is the evaluation map $<j(x,u),f>=(u|f(x))$ for $(x,u)\in E$ and $f\in C(X,E)$, is a homeomorphism onto its image.
\end{prop}
 
 \begin{proof} The map $(id_X,j): E\ra X\times {\mathcal E}^*$ is injective: consider $(x,u)$ and $(x',u')$ in $E$. If $x\not=x'$, they have distinct images. If $x=x'$ and $u\not=u'$, there exists $v\in E_x$ such that $(u|v)\not=(u'|v)$. There exists $f\in C(X,E)$ such that $f(x)=v$. Then $<j(x,u),f>\not=<j(x,u'),f>$ and $j(x,u)\not=j(x,u')$. The map is continuous with respect to the weak topologies: if the net $(x_i,u_i)$ converges to $(x,u)$ in $E_\sigma$, then $x_i$ converges to $x$ in $X$; for $f\in C(X,E)$, $<j(x_i,u_i),f>=(u_i|f(x_i))$ converges to $(u|f(x))=<j(x,u),f>$ by definition of the weak topology of $E$. Conversely, if $x_i$ converges to $x$ in $X$ and $j(x_i,u_i)$ converges to $j(x,u)$ in ${\mathcal E}^*_\sigma$, then by definition, $(x_i,u_i)$ converges to $(x,u)$ in $E_\sigma$.
 \end{proof}

\begin{defn} We say that a subset $A$ of the bundle space $E$ of a Hilbert bundle is bounded if the norm function is bounded on $A$.
\end{defn}

\begin{lem}\label{joint continuity} Let $E$ be a Hilbert bundle over a locally compact space $X$. Assume that the net $(x_i,v_i)$ (based on some directed set $J$) is bounded and converges to $(x,v)$ in $E_\sigma$ and that the net $(x_i,e_i)$ (based on the same $J$) converges to $(x,e)$ in $E$. Then, the net $(v_i|e_i)$ converges to $(v|e)$.
\end{lem}

\begin{proof} Choose a compact neighborhood $V$ of $x$. Choose a continuous section $f:V\ra E$ such that $f(x)=e$. Assuming that $\|v_i\|\le a$, we have
$$\begin{array}{cc}
(v_i|e_i)-(v|e)&=(v_i|e_i-f(x_i)) +(v_i|f(x_i)-(v|f(x))\\
|(v_i|e_i)-(v|e)|&\le a\|e_i-f(x_i)\| +|(v_i|f(x_i)-(v|f(x))|.\\
\end{array}$$ 
By continuity of the addition in $E$, $\|f(x_i)-e_i\|$ tends to 0. By definition of the weak convergence, $|(v_i|f(x_i)-(v|f(x))|$ tends also to 0.
\end{proof}

Note that this lemma gives another definition of the weak convergence of a bounded net $(x_i,v_i)$. The next lemma generalizes a well-known characterization of strongly convergent nets in Hilbert spaces.

\begin{lem}\label{weak/strong} Let $E$ be a Hilbert bundle over a locally compact space $X$. Let $(x_i,u_i)$ be a net in $E$ and let $(x,u)$ be an element of $E$. Then \tfae
\begin{enumerate}
\item $(x_i,u_i)\to (x,u)$ strongly;
\item $(x_i,u_i)\to (x,u)$ weakly and $\|u_i\|\to \|u\|$; 
\end{enumerate}
\end{lem}

\begin{proof} The implication $(i)\Rightarrow (ii)$ is clear. Suppose that $(ii)$ holds. Choose a compact neighborhood $V$ of $x$ and choose a continuous section $f:V\ra E$ such that $f(x)=u$. By definition of the weak convergence, $(f(x_i)|u_i)\to (f(x)|u)=\|u\|^2$. Therefore:
$$\|f(x_i)-u_i\|^2=\|f(x_i)\|^2+\|u_i\|^2-2{\rm Re}(f(x_i)|u_i)$$
tends to 0. According to \cite[Proposition 13.12]{fd:representations}, this implies that $(x_i,u_i)$ tends to $(x,u)$ strongly.
\end{proof}

We are interested in continuity properties of sections of a Hilbert bundle $E$. Given an arbitrary section $f:X\ra E$, we can define a map $<f|$ which sends a section $g$ to the scalar function $<f,g>(x)=(f(x)|g(x))$ and a map $\tilde f: E\ra \C$ which sends $(x,u)\in E$ to $(f(x)|u)$. We have:

\begin{prop}\label{weak section} Let $E$ be a Hilbert bundle over a compact metrizable space $X$. Let $f:X\ra E$ be a section. Then \tfae
\begin{enumerate}
\item $f$ is weakly continuous;
\item $<f|$ sends $C(X,E)$ into $C(X)$;
\item $\tilde f: E\ra \C$ is continuous with respect to the strong topology.
\end{enumerate}
\end{prop}

\begin{proof} The equivalence of $(i)$ and $(ii)$ results directly from the definition of the weak topology. Let us show that $(i)\Rightarrow (iii)$. Let $f:X\ra E$ be weakly continuous. According to \cite[Proposition 1.1.9]{laf:gpd} (this is the only place where the metrizability of $X$ is used), $\|f(x)\|_x$ is bounded on $X$. Suppose that $(x_i,u_i)$ converges  to $(x,u)$ in $E$. According to \lemref{joint continuity}, $(f(x_i)|u_i)$ converges to $(f(x)|u)$. This proves the continuity of $\tilde f$. The implication $(iii)\Rightarrow (ii)$ is clear, because, for $g\in C(X,E)$, $<f,g>=\tilde f\circ g$.
\end{proof}

There is an analogous characterization of (strongly) continuous sections:

\begin{prop}\label{strong section} Let $E$ be a Hilbert bundle over a compact space $X$. Let $f:X\ra E$ be a section. Then \tfae
\begin{enumerate}
\item $f$ is strongly continuous;
\item $<f|$ is an adjointable $C(X)$-linear map from $C(X,E)$ into $C(X)$;
\item $\tilde f: E\ra \C$ is continuous with respect to the weak topology. 
\end{enumerate}
\end{prop}

\begin{proof} The only possible candidate for the adjoint of $<f|: C(X,E)\ra C(X)$ is the map $|f>:C(X)\ra C(X,E)$ sending $h\in C(X)$ to $hf$. This map exists if and only if $f\in C(X,E)$. This proves the equivalence of $(i)$ and $(ii)$. The implication $(i)\Rightarrow (iii)$ results from the definition of the weak topology. Suppose that $(iii)$ holds. Let us show that $f$ is strongly continuous. The continuity of $\tilde f: E_\sigma\ra \C$ implies the continuity of $\tilde f: E\ra \C$; according to \propref{weak section}, $f$ is weakly continuous. Therefore, if $x_i$ tends to $x$, then $f(x_i)$ tends to $f(x)$ weakly. But then, $\|f(x_i)\|^2=\tilde f(f(x_i))$ tends to $\tilde f(f(x))=\|f(x)\|^2$. According to \lemref{weak/strong}, this implies that  $f(x_i)$ tends to $f(x)$ strongly.
\end{proof}

Assume that $X$ is compact metrizable. The space $C(X,E_\sigma)$ of weakly continuous sections can be identified with the space of $C(X)$-linear bounded maps from $C(X,E)$ to $C(X)$. It agrees with $C(X,E)$ if and only if $C(X,E)$ is a self-dual Hilbert module. This is the case for example when $E$ is a vector bundle in the usual sense (i.e. locally trivial finite dimensional) but also when $X$ is reduced to a point.

\begin{prop} Let $E$ be a Hilbert bundle over a topological space $X$. For $0<R<\infty$, we define the cylinder $C_R=\{(x,u)\in E: \|u\|_x\le R\}$. Then
\begin{enumerate}
\item if $X$ is compact, $C_R$ is a compact subset of $E_\sigma$;
\item if $X$ is locally compact, $C_R$ is a closed subset of $E_\sigma$.

\end{enumerate}

\end{prop}

\begin{proof} Let us assume that $X$ is compact. Since $(id_X,j)(C_R)$ is contained in $X\times B_R$, where $B_R$ is the closed ball of radius $R$ of ${\mathcal E}^*$ which is $*$-weakly compact, it suffices to show that $(id_X,j)(C_R)$ is  closed in $X\times {\mathcal E}^*_\sigma$. Let  $(x_i,u_i)$ be a net in $C_R$ such that $(x_i,j(x_i,u_i))$ converges to $(x,\varphi)$, where $\varphi\in {\mathcal E}^*$. Then $x_i$ converges to $x$. Let us show that $\varphi=j(x,u)$ for some $u\in E_x$. For that, it suffices to show that $<\varphi,f>=0$ for all  $f\in C(X,E)$ such that $f(x)=0$. Let $f$ and $\epsilon>0$ be given. There exists a neighborhood $V$ of $x$ such that $\|f(y)\|\le \epsilon/R$ for all $y\in V$. For $i$ large enough, $x_i$ belongs to $V$, thus 
$$|(u_i | f(x_i))|\le \|u_i\|\|f(x_i)\|\le R(\epsilon/R)$$
which implies $|<\varphi,f>|\le \epsilon$ and $<\varphi, f>=0$. 

Let us assume that $X$ is locally compact. Let  $(x_i,u_i)$ be a net in $C_R$ which converges weakly to $(x,u)$. The point $x$ has a compact neighborhood $V$. For $i$ large enough, $x_i$ belongs to $V$ and $(x_i,u_i)$ belongs to the cylinder $C_R(E_V)$ of the reduction $E_V$ of $E$ to $V$, which is compact. Since a net which is weakly convergent in $E_V$ is also weakly convergent in $E$, the limit $(x,u)$ belongs to the cylinder $C_R(E_V)$.
\end{proof}

\begin{cor}\label{compact} Let $E$ be a Hilbert bundle over a compact space $X$. Then bounded subsets are relatively compact in $E_\sigma$.
\end{cor}

\begin{defn} We say that a subset $A$ of the bundle space $E$ of a Hilbert bundle over $X$ is fiberwise convex if for every $x\in X$, $A_x\defequal A\cap E_x$ is convex.
\end{defn}

\begin{prop} Let $E$ be a Hilbert bundle over a locally compact space $X$. Given a bounded subset $A$ of $E$, there exists a smallest fiberwise convex and weakly closed subset of $E$ containing $A$. It is bounded. This subset will be called the convex hull of $A$ and denoted by ${\rm conv}(A)$.
\end{prop}

\begin{proof} By assumption $A$ is contained in some cylinder $C_R$, which is fiberwise convex and weakly closed. The intersection of all fiberwise convex and weakly closed subsets containing $A$, which is fiberwise convex and weakly closed, is the sought-after set.
\end{proof}

Given a continuous map $p:Y\ra X$ and a Hilbert bundle $\pi:E\ra X$, the pull-back bundle $p^*E$ is the Hilbert bundle over $Y$ defined as
$$p^*E=\{(y, u)\in Y\times E: p(y)=\pi(u)\}.$$
Its bundle projection is the restriction of the first projection. Its topology is the subspace topology. We denote by $P:p^*E\ra E$ the map defined by $P(y,u)=(p(y),u)$. We leave to the reader to check that, when $X$ and $Y$ are locally compact spaces, $(p^*E)_\sigma=p^*E_\sigma$, i.e. the weak topology of $p^*E$ agrees with the subspace topology of $Y\times E_\sigma$.

\begin{prop}\label{conv} Let $X,Y$ be locally compact spaces and let $p:Y\ra X$ be continuous and open. Let $E$ be a Hilbert bundle over $X$ and let $F=p^*E$ be its pull-back over $Y$. If $A$ is a bounded subset of $E$, then $B=P^{-1}(A)$ is a bounded subset of $F$ and ${\rm conv}(B)=P^{-1}({\rm conv}(A))$.
\end{prop}

\begin{proof} It is clear that the range of the norm function is the same on $A$ and on $B$. Suppose that $C$ is a fiberwise convex and weakly closed set containing $B$. Consider its contraction
$$C'=\{(y,u)\in F: \quad p(y')=p(y) \Rightarrow (y',u)\in C\}.$$
It is  fiberwise convex, because an intersection of convex sets is convex. I claim that it is weakly closed.  Suppose that the net $(y_i, u_i)$ in $C'$ converges to $(y,u)$. Let $y'\in Y$ such that $p(y')=p(y)$. Since $p$ is open, there is a net $(y'_i)$ converging to $y'$ such that $p(y'_i)=p(y_i)$ for all $i$. Then $(y'_i,u_i)$ belongs to $C$. Let us show that $(y'_i,u_i)$ converges to $(y',u)$. Let $V$ be a compact neighborhood of $y'$. According to \cite[Proposition II.14.1]{fd:representations}, the sums of continuous sections of the form $f(y)=h(y)g\circ p(y)$ where $h\in C(V)$ and $g\in C(p(V),E)$ are dense in $C(V,p^*E)$ in the sup-norm topology; therefore, it suffices to check the convergence on such a section $f=h g\circ p$. Then
$(u_i|f(y'_i))=h(y'_i)(u_i|g\circ p(y_i))$ converges to $h(y')(u| g\circ p(y))=(u|f(y'))$. Since $C$ is weakly closed, $(y',u)$ belongs to $C$. Therefore $(y,u)$ belongs to $C'$ as claimed. Note that $P(C')$ is fiberwise convex and bounded. It is also weakly closed: let $(y_i, u_i)$ be a net in $C'$ such that $(p(y_i), u_i)$ converges to $(x, u)$. Let $y\in Y$ such that $p(y)=x$. There exists a net $(y'_i)$ converging to $y$ such that $p(y'_i)=p(y_i)$ for all $i$. Then $(y'_i,u_i)$ belongs also to $C'$ and the net $(y'_i, u_i)$ converges to $(y,u)$. Thus $(y,u)$ belongs to $C'$ and $(x,u)$ belongs to $P(C')$. Since $P(C')$ contains $A$, it contains ${\rm conv}(A)$.  Therefore $C\supset P^{-1}(P(C')$ contains $P^{-1}({\rm conv}(A))$. This gives the inclusion
${\rm conv}(B)\supset P^{-1}({\rm conv}(A))$. On the other hand $P^{-1}({\rm conv}(A))$ is a subset of $F$ containing $B$ which is fiberwise convex and weakly closed. Hence it contains ${\rm conv}(B)$.

\end{proof}

Let us assume now that $E$ is a $G$-Hilbert bundle, where $G$ is a topological groupoid over $X$. Recall that we assume that for all $\gamma\in G$, $L(\gamma):E_{s(\gamma)}\ra E_{r(\gamma)}$ is a linear isometry and that the action map $G*E\ra E$ is continuous. We are also given a continuous cocycle $c:G\ra r^*E$ which defines the affine isometric action of $G$ on $E$ given by:

$$\gamma u=L(\gamma)u+c(\gamma),\quad\forall (\gamma,u)\in G*E.$$

\begin{prop}\label{weak continuity} Let $(\gamma_i,u_i)$ be a net  converging to $(\gamma,u)$ in $G*E_\sigma$. If $(\|u_i\|)$ is bounded, then $(\gamma_i u_i)$ converges to $\gamma u$ in $E_\sigma$.
\end{prop}

\begin{proof}
Let $f\in C(X,E)$. We have
$$\begin{array}{ccc}
(\gamma_iu_i|f\circ r(\gamma_i))&=&(L(\gamma_i)u_i|f\circ r(\gamma_i))+(c(\gamma_i)|f\circ r(\gamma_i))\\
&=&(u_i|L(\gamma^{-1}_i)f\circ r(\gamma_i))+(c(\gamma_i)|f\circ r(\gamma_i))\\
\end{array}$$
By continuity of the action, $L(\gamma^{-1}_i)f\circ r(\gamma_i)$ tends to $L(\gamma^{-1})f\circ r(\gamma)$ in $E$ and by \lemref{joint continuity}, $(u_i|L(\gamma^{-1}_i)f\circ r(\gamma_i))$ tends to $(u|L(\gamma^{-1})f\circ r(\gamma))$. By joint continuity of the scalar product in $E$, $(c(\gamma_i)|f\circ r(\gamma_i))$ tends to $(c(\gamma)|f\circ r(\gamma))$. Hence the result.
\end{proof}

We say that a subset $A$ of $E$ is invariant if for every $(\gamma,u)\in G\times A$ such that $s(\gamma)=\pi(u)$, $\gamma u$ belongs to $A$. Let us introduce the pull-back $s^*E=G*E$ of $E$ along the source map $s: G\ra X$ and the map $W:s^*E\ra s^*E$ defined by $W(\gamma, u)=(\gamma^{-1},\gamma u)$. This map is the fundamental involution of the action which is ubiquitous in the theory of quantum groups. For its use in a similar context, see \cite[Section 4]{leg:KKG}. Let also define the map $S: s^*E\ra E$ by $S(\gamma,u)=(s(\gamma),u)$. Then, we have the following convenient criterium for invariance.
\begin{lem} A subset $A$ of $E$ is invariant if and only if $W(S^{-1}(A))=S^{-1}(A)$.
\end{lem}

\begin{prop}\label{inv conv} Let $E$ be a $G$-Hilbert bundle, where $G$ is a locally compact groupoid. We consider the affine isometric action of $G$ on $E$ defined by a continuous cocycle $c:G\ra r^*E$. Let $A$ be a bounded subset of $E$. If $A$ is invariant, then its convex hull ${\rm conv}(A)$ is also invariant.
\end{prop}

\begin{proof} The fundamental involution $W:s^*E\ra s^*E$ sends the fibre $E_\gamma=E_{s(\gamma)}$ onto the fibre $E_{\gamma^{-1}}=E_{r(\gamma)}$ through the affine map $u\mapsto \gamma u$. Therefore, it respects fiberwise convex sets. It also sends bounded weakly closed sets onto weakly closed sets. Suppose indeed that $B\subset s^*E$ is bounded and weakly closed. Let $(\gamma_i, u_i)$ be a net in $W(B)$ converging weakly to $(\gamma, u)$. Then $(\gamma_i)$ converges to $\gamma$; in particular, the net $(\|c(\gamma_i)\|)$ is bounded. The net $(\|\gamma_i u_i\|)$ is also bounded because $(\gamma_i^{-1},\gamma_iu_i)$ belongs to the bounded set $B$. Since
$u_i=L(\gamma_i^{-1})(\gamma_i u_i-c(\gamma_i))$, the net $(\|u_i\|)$ is bounded. According to \propref{weak continuity}, $(\gamma_i u_i)$ converges to $\gamma u$ in $E_\sigma$, therefore $(\gamma_i^{-1},\gamma_i u_i)$ converges to $(\gamma^{-1},\gamma u)$ in $(s^*E)_\sigma$. Since $B$ is weakly closed, $(\gamma^{-1},\gamma u)$ belongs to $B$ and $(\gamma, u)$ belongs to $W(B)$. Therefore, if $B$ is a bounded subset of $s^*E$, $W({\rm conv}(B))$ is a fiberwise convex weakly closed set containing $W(B)$. This shows that the convex hull ${\rm conv}(W(B))$ exists and is contained in $W({\rm conv}(B))$. If moreover $W(B)$ is bounded, the same argument gives ${\rm conv}(B)\subset W({\rm conv}(W(B)))$, hence the equality  ${\rm conv}(W(B))=W({\rm conv}(B))$. Applied to $B=S^{-1}(A)$, where $A$ is an invariant bounded subset of $E$, this gives
${\rm conv}(W(S^{-1}(A)))=W({\rm conv}(S^{-1}(A)))$. Since $W(S^{-1}(A))=S^{-1}(A)$, we get ${\rm conv}(S^{-1}(A))=W({\rm conv}(S^{-1}(A)))$. Since, according to \propref{conv}, ${\rm conv}(S^{-1}(A))=S^{-1}({\rm conv}(A))$, this gives the invariance of ${\rm conv}(A)$.
\end{proof}

 We are now ready to prove the existence of a weakly continuous equivariant section.
 
 \begin{thm}\label{weak Renault} Let $G$ be a minimal locally compact groupoid on a compact space $X$, let $E$ be a continuous $G$-Hilbert bundle and let $c:G\ra r^*E$ be a continuous cocycle.  If there exists $x\in X$ such that $\|c(G_x)\|$ is bounded, then there exists a weakly continuous section $f:X\ra E$  such that
 $$\forall \gamma\in G,\quad c(\gamma)=f\circ r(\gamma)-L(\gamma)f\circ s(\gamma).$$
\end{thm}

\begin{proof} This part follows closely \cite{cnp:bdd orbits}. We consider the affine action of $G$ on $E$ defined by the cocycle $c$. The assumption is the existence of a bounded orbit $A$. The conclusion is the existence of a weakly continuous equivariant section. \propref{inv conv} gives the existence of a non-empty weakly closed, fiberwise convex, invariant and bounded subset, namely ${\rm conv}(A)$. Since $X$ is compact, according to \corref{compact}, this set is weakly compact. The family of all weakly compact fiberwise convex and invariant non-empty subsets of $E$ ordered by inclusion is inductive. By Zorn's lemma, there exists a minimal weakly compact fiberwise convex invariant non-empty subset $M$. Then $\pi(M)$ is a non-empty closed invariant subset of $X$. By minimality of $G$, $\pi(M)=X$. This says that for all $x\in X$, the fiber $M_x=M\cap E_x$ has at least one element. We are going to show that $M_x$ has at most one element. This is a classical trick which uses the uniform convexity of the Hilbert spaces $E_x$. Hilbert spaces have the uniform convexity module $\delta=\delta(\epsilon)=\sqrt{(1-{1\over 4}\epsilon^2)}$. Recall that this means that for $u_1,u_2\in E_x$,
$$\|u_1\|\le 1, \|u_2\|\le 1, \|u_1-u_2\|\ge \epsilon\quad \Longrightarrow\quad \|{1\over 2}(u_1+u_2)\|\le 1-\delta.$$
We fix $\epsilon>0$. We let $R=\sup_{u\in M}\|u\|$ and choose $u\in M$ such that $\|u\|>(1-\delta^2)R$. We choose $f\in C(X,E)$ such that $f\circ\pi(u)=u/\|u\|$. Let $V=\{y\in X: \|f(y)\|<1+\delta\}$. It is an open neighborhood of $\pi(u)$ in $X$.

Let $x$ be an arbitrary point in $X$ and assume that $u_1,u_2$ belong to the fiber $M_x$. The midpoint $m={1\over 2}(u_1+u_2)$ also belongs to the convex set $M_x$. Let us show that the orbit of $m$ meets the non-empty weakly open set
$$U=\{(y,e)\in E: y\in V,\quad |(e|f(y)|>(1-\delta^2)R\}.$$
If not, this orbit would be contained in the fiberwise convex weakly closed subset $M\setminus U$ and so would be its convex hull. This would contradict the minimality of $M$ because $M\setminus U$ is strictly contained in $M$: it does not contain $u$. Thus, there exists $\gamma\in G_x$ such that $\gamma m\in U$. Then, we must have $\|\gamma m\|> (1-\delta)R$. Since $G$ acts by affine transformations, $\gamma m$ is the midpoint of $\gamma u_1$ and $\gamma u_2$. Moreover, since these points belong to $M$, they satisfy $\|\gamma u_i\|\le R$. The above uniform convexity condition implies that $\|\gamma u_1-\gamma u_2\|<\epsilon R$. Since $G$ acts by isometries, this implies that $\|u_1-u_2\|<\epsilon R$. Since $\epsilon$ is arbitrary, this implies $u_1=u_2$.

Thus, the restriction of the bundle projection $\pi_{|M}: M\ra X$ is a bijection. Since it is weakly continuous and $M$ is compact with respect to the weak topology of $E$, its reciprocal map $f: X\ra M$ is weakly continuous. The invariance of $M$ says exactly that for all $\gamma\in G$, $\gamma f\circ s(\gamma)=f\circ r (\gamma)$.
\end{proof}

\section{Continuity of $G$-equivariant weakly continuous sections.}

As before, we assume that the base space $X$ of the Banach bundle $E$ is locally compact.
We are going to show that an equivariant weakly continuous section is norm continuous. Again, it is an adaptation of the proof given in \cite{cnp:bdd orbits}. However, it is no longer possible to use the oscillation function as in \cite{cnp:bdd orbits}: for a section $f:X\ra E$ of a Banach bundle, one cannot compare directly the vectors $f(x)$ and $f(y)$ for $x\not=y$ since they  belong to different spaces.

\begin{defn}\label{module of continuity} We define the module of continuity of a section $f:X\ra E$ at $x$ as
$\omega(x)=\inf\{\sup_{y\in V}\|f(y)-g(y)\|\}$, where the infimum is taken over all pairs $(V,g)$, where $V$ is an open neighborhood of $x$ and $g:V\ra E$ is a continuous section over $V$.
\end{defn}

\begin{prop}
Let $f:X\ra E$ be a section of a Banach bundle $\pi:E\ra X$ and $x\in X$. Then \tfae
\begin{enumerate}
\item $f$ is continuous at $x$;
\item its module of continuity at $x$ vanishes.
\end{enumerate}
\end{prop}

\begin{proof} 
This is a corollary of \cite[Proposition II.13.12]{fd:representations}; this can also be seen directly from the comments following \defnref{Banach bundle}.
\end{proof}

\begin{prop} Let $f:X\ra E$ be a section with module of continuity $\omega$ and let $U=\{x\in X: \omega(x)<\epsilon\}$, where $\epsilon>0$. Then,
\begin{enumerate}
\item $U$ is an open subset of $X$.
\item If $f$ is $G$-equivariant with respect to a continuous affine isometric action of a topological groupoid $G$, then $U$ is $G$-invariant.
\end{enumerate}
\end{prop}

\begin{proof} The first assertion results directly from the definition of the module of continuity $\omega(x)$: if $x\in U$, there is a pair $(V,g)$ where $V$ is an open neighborhood of $x$ such that $\|f(y)-g(y)\|<\epsilon$ for all $y\in V$. Then $V$ is contained in $U$.  

Let us prove $(ii)$. Let $\underline\gamma\in G$. We assume that $s(\underline\gamma)\in U$ and we will show that $r(\underline\gamma)\in U$. We fix $\epsilon'$ such that $\omega(s(\underline\gamma))<\epsilon'<\epsilon$. Then there exists an open neighborhood $V$ of $s(\underline\gamma)$ and a continuous section $g:V\ra E$ such that for all $y\in V$, $\|f(y)-g(y)\|<\epsilon'$. Let us define $\tilde h:G_V=s^{-1}(V)\ra r^*E$ by $\tilde h(\gamma)=\gamma g\circ s(\gamma)$. Let $\delta=\epsilon-\epsilon'$. Since $\tilde h$ is a continuous section of the pull-back bundle $r^*E$, according to \cite[Section 5]{dd:dixmier-douady}, there exists an open neighborhood $S\subset G_V$ of $\underline\gamma$ and a continuous section $h:r(S)\ra E$ such that for all $\gamma\in S$, we have $\|\tilde h(\gamma)-h\circ r(\gamma)\|<\delta$. Then, for all $x\in W=r(S)$, we have:
$$\begin{array}{ccc}
\|f(x)-h(x)\|&=&\|f\circ r(\gamma)-h\circ r(\gamma)\| \\
&=&\|f\circ r(\gamma)-\tilde h(\gamma)\|+ \|\tilde h(\gamma)-h\circ r(\gamma)\|\\
&=&\|\gamma f\circ s(\gamma)-\gamma g\circ s(\gamma)\|+ \delta\\
&=&\|f\circ s(\gamma)-g\circ s(\gamma)\| +\delta\\
&<&\epsilon'+\delta=\epsilon .\\
\end{array}$$
This shows that $\omega(r(\underline\gamma))<\epsilon$ and $r(\underline\gamma)\in U$.

\end{proof}

\begin{cor}Let $f:X\ra E$ be a $G$-equivariant section of a Banach bundle $E$ endowed with an affine isometric continuous action of a topological groupoid $G$. Then the set of  points of continuity of $f$ is a countable intersection of open invariant subsets.
\end{cor}

\begin{proof} We define $U_n=\{u\in E: \omega(u)<1/n\}$. Then, we just observe that the set of points of continuity of $f$ is the intersection of the $U_n$ 's.
\end{proof}

\begin{cor}\label{inv} Let $f:X\ra E$ be a $G$-equivariant section of a Banach bundle $E$ endowed with an affine isometric continuous action of a topological groupoid $G$. If $G$ is minimal, either $f$ is continuous or has no point of continuity.
\end{cor}

\begin{proof} If $f$ has at least one point of continuity, the $U_n$'s are non-empty. By minimality, they are all equal to $X$. Therefore the set of points of continuity of $f$ is $X$.
\end{proof}

Let us recall a general property of Hilbert C*-modules over C*-algebras.

\begin{prop}\label{approximation}\cite[Theorem 3.1]{ble:C*-modules} Let $\mathcal E$ be a  C*-module over a C*-algebra $A$. Then there exists a directed set $I$, a net of integers $(n_i)$ and nets of contractive $A$-linear maps $\varphi_i:{\mathcal E}\ra A^{n_i}$, of the form $\varphi_i(f)=(<g_{i,1},f>,\ldots,<g_{i,n_i},f>)$ with $g_{i,1},\ldots,g_{i,n_i}\in {\mathcal E}$, and $\psi_i: A^{n_i}\ra{\mathcal E}$ such that for all $f\in{\mathcal E}$, $\psi_i\circ\varphi_i(f)$ tends to $f$. Moreover, if $\mathcal E$ is countably generated, one can choose $I=\N$.
\end{prop}

Note that the maps $\varphi_i$ and $\psi_i$ of the above proposition are adjointable maps from a C*-module to another C*-module.

\begin{cor}\label{weak-strong} Let $E\ra X$ be a Hilbert bundle. Assume that $X$ is compact and that the bundle space $E$ is second countable. Then the set of points of continuity of each weakly continuous section $f:X\ra E$  is a dense $G_\delta$.
\end{cor}

\begin{proof} We apply \propref{approximation} to the C*-module ${\mathcal E}=C(X,E)$ over the C*-algebra $A=C(X)$. According to \cite[Proposition II.13.21]{fd:representations}, it is countably generated. Because the functor $E\mapsto C(X,E)$ gives an equivalence between the category of Hilbert bundles over $X$  and the category of Hilbert C*-modules over $C(X)$, we obtain a sequence of continuous bundle maps $\varphi_i:E\ra X\times \C^{n_i}$ and $\psi_i: X\times \C^{n_i}\ra E$ such that for all $u\in E$, $\psi_i\circ\varphi_i(u)$ tends to $u$. 

Let $f$ be a weakly continuous section of $E$. Then $\varphi_i\circ f: X\ra X\times \C^{n_i}$, which is of the form $x\mapsto (x,((g_{i,1}(x)|f(x)),\ldots,(g_{i,n_i}(x),f(x)))$, where $g_{i,1},\ldots,g_{i,n_i}\in C(X,E)$, is continuous. Therefore $f$ is the pointwise limit of the sequence of the continuous sections $f_i=\psi_i\circ(\varphi_i\circ f)$. Our assumption implies that $X$ is compact and second countable, hence metrizable. The space $E$ is also metrizable since, according to \cite{dd:dixmier-douady}, it can be embedded into a trivial bundle $X\times F$, where $F$ is a Hilbert space. According to a theorem of Baire, the set of points of continuity of $f:X\ra E$  is a dense $G_\delta$.
\end{proof}

\begin{cor}  Let $E\ra X$ be a Hilbert bundle endowed with an affine isometric continuous action of a topological groupoid $G$. Assume that $X$ is compact, the bundle space $E$ is second countable and $G$ is minimal. Then each $G$-equivariant section $f:X\ra E$ which is weakly continuous is necessarily strongly continuous.
\end{cor}

\begin{proof} \corref{weak-strong} shows that $f$ has at least one point of continuity. \corref{inv} says then that $f$ is continuous at all $x\in X$.
\end{proof}

\section{ Concluding remarks}

C. Anantharaman-Delaroche gives in \cite[Theorem 3.19]{ana:T} a measure theoretic version of \thmref{main}. Her proof is based on the ``lemma of the centre'' \cite[Lemma 2.2.7]{bhv:T}. Let us assume that $E$ is a continuous $G$-Hilbert bundle and $c:G\ra r^*E$ is a bounded continuous cocycle. Then we can define for all $x\in X$, $f(x)$ as the centre of $c(G^x)$. Then, since $L(\gamma)$ is an isometry from $E_{s(\gamma)}$ to $E_{r(\gamma)}$, we have $c(\gamma)=f\circ r(\gamma)-L(\gamma)f\circ s(\gamma)$ for all $\gamma\in G$. In the measure theoretical framework, $f$ is measurable; however, its continuity is problematic (see \cite[Example 17]{cnp:bdd orbits}).

J.-L. Tu shows in \cite[3.3]{tu:conjecture} that, just as in the case of groups, every conditionally negative type continuous function $\psi: G\ra \R$ on a topological groupoid $G$ is of the form $\psi(\gamma)=\|c(\gamma)\|^2$, where $c: G\ra r^*E$ is a continuous cocycle of a continuous $G$-Hilbert bundle $E$. \thmref{main} shows that the condition that every conditionally negative type continuous function is bounded has the same cohomological interpretation for groupoids as for groups.

\vskip 5mm
{\it Acknowledgements.} I thank C.~Anantharaman-Delaroche and E.~Blanchard for their help in eliminating some obscurities in a preliminary draft of the manuscript.


\vskip3mm




\begin{thebibliography}{10}

%\bibitem[label]{cle} Auteur, TITRE, editeur, annee
\bibitem{ana:T} C. Anantharaman-Delaroche, {\it Cohomology of property T groupoids and applications}, Ergod. Th. \& Dynam. Sys. \textbf{25}
(2005), 465--471.

\bibitem{ble:C*-modules} D. Blecher, {\it A new approach to C*-modules}, Math. Ann. \textbf{307} (1997), 253--290.

\bibitem{bgm:fixed} U. Bader, T. Gelander and N. Monod, {\it A fixed point theorem for $L^1$spaces,} Math arXiv:\textbf{1012.1488v1}, (2010).

\bibitem{bhv:T} B. Bekka, P. de la Harpe and A. Valette, {\it Kazhdan's Property (T)}, 
Cambridge University Press, Cambridge, 2008.

\bibitem{bbki:topologie} N. Bourbaki, {\it Topologie g\'en\'erale, Ch. 1 \`a 4}, 
Diffusion C.C.L.S., Paris, 1971.


\bibitem{cnp:bdd orbits} D. Coronel, A. Navas and M. Ponce, {\it On bounded cocycles of isometries over a minimal dynamics,} Math arXiv:\textbf{1101.3523v4}, (2011).

\bibitem{dix:C*} J. Dixmier, {\it Les C*-alg\`ebres et leurs repr\'esentations}, Gauthier-Villars, Paris, 1969.

\bibitem{dd:dixmier-douady} J. Dixmier and A. Douady, {\it Champs continus d'espaces hilbertiens et de C*-alg\`ebres}, Bull. Soc. Math. France \textbf{91} (1963), 227--283.

\bibitem{er:semigroups} R. Exel and J. Renault, {\it Semigroups of local homeomorphisms}, Ergod. Th. \& Dynam. Sys. \textbf{27} (2007), 1737--1771.

\bibitem{fd:representations}J. Fell and R. Doran, {\it Representations of $*$-algebras, locally compact groups, and Banach $*$-algebraic bundles} vol. 1, Academic Press,  1988.

\bibitem{gh:top dyn} W. Gottschalk and G. Hedlund, {\it Topological dynamics}, Amer. Math. Soc., Providence, R.I., 1955.

\bibitem{laf:gpd} V. Lafforgue, {\it K-th\'eorie bivariante pour les alg\`ebres de Banach, groupo\"ides et conjecture de Baum-Connes, avec un appendice d'Herv\'e Oyono-Oyono}, J. Inst. Math. Jussieu, \textbf{6} no.3 (2007), 415--451.

\bibitem{leg:KKG} P.-Y. Le Gall, {\it Th\'eorie de Kasparov \'equivariante et groupo\"ides I}, K-Theory, \textbf{16} (1999), 361--390.

\bibitem{mrw:Morita} P. Muhly, J. Renault and D. Williams, {\it Equivalence and isomorphism for groupoid $C^*$-algebras},
J. Operator Theory, \textbf{17} (1987), 3--22.

\bibitem{ren:approach} J.~Renault: {\it A groupoid approach to
$C^*$-algebras}, Lecture Notes in Mathematics, Vol.~{\bf 793}
Springer-Verlag Berlin, Heidelberg, New York (1980).

\bibitem {tu:conjecture}J.-L. Tu, {\it Conjecture de Baum-Connes pour les feuilletages moyennables}, K-Theory, \textbf{17} (1999), 215--264. 

\bibitem {tu:cohomology}J.-L. Tu, {\it Groupoid cohomology and extensions}, Trans. Amer. Math. Soc. \textbf{358} (2006), 4721--4747. 

\end{thebibliography}
\end{document}